\documentclass[onefignum,onetabnum]{siamart220329}

\usepackage{amsmath}
\usepackage{amssymb}
\usepackage{algorithm}
\usepackage{algorithmic}
\usepackage{subfig}
\usepackage{graphicx}
\usepackage{mdframed}
\usepackage[utf8]{inputenc}

\usepackage{todonotes}

\DeclareMathAlphabet{\mathbf}{OT1}{cmr}{bx}{it}

\newcommand{\vb}{\mathbf b}

\newcommand{\ve}{\mathbf e}
\newcommand{\vf}{\mathbf f}
\newcommand{\vg}{\mathbf g}

\newcommand{\vr}{\mathbf r}
\newcommand{\vs}{\mathbf s}
\newcommand{\vt}{\mathbf t}

\newcommand{\vv}{\mathbf v}

\newcommand{\vx}{\mathbf x}
\newcommand{\vy}{\mathbf y}
\newcommand{\vz}{\mathbf z}

\newcommand{\vnull}{\boldsymbol{0}}

\newcommand{\vG}{\mathbf G}

\newcommand{\vR}{\mathbf R}

\newcommand{\vV}{\mathbf V}
\newcommand{\vW}{\mathbf W}

\renewcommand{\d}{\,\mathrm d}
\newsiamremark{remark}{Remark}
\newsiamremark{hypothesis}{Hypothesis}
\crefname{hypothesis}{Hypothesis}{Hypotheses}
\newsiamthm{claim}{Claim}

\newcommand{\sign}{\mbox{sign}}

\newcommand{\quadp}{\textrm{quad}}
\newcommand{\cut}[1]{}

\title{Krylov subspace recycling for matrix functions
}
\author{Liam Burke\thanks{Corresponding author, School Of Mathematics, Trinity College Dublin, College Green, Dublin 2, Ireland (\email{Email: burkel8@tcd.ie})}
\and Andreas Frommer\thanks{ Department of Mathematics, Bergische Universit\"at Wuppertal, 42097 Wuppertal, Germany}
\and Gustavo Ramirez-Hidalgo\footnotemark[2]
\and Kirk M. Soodhalter\thanks{School Of Mathematics, Trinity College Dublin, College Green, Dublin 2, Ireland}
}

\usepackage{amsopn}

\ifpdf
\hypersetup{
  pdftitle={Krylov subspace recycling for matrix functions},
  pdfauthor={Liam Burke, Andreas Frommer, Gustavo Ramirez Hidalgo and Kirk M. Soodhalter}
}
\fi

\begin{document}

% \input{replies}
% \newpage

\maketitle

% REQUIRED
\begin{abstract}
We derive an augmented Krylov subspace method with subspace recycling for computing a sequence of matrix function applications on a set of vectors. The matrix is either fixed or changes as the sequence progresses. We assume consecutive matrices are closely related, and that the vectors are available only in sequence rather than simultaneously. We present three versions of the method with different practical implementations and demonstrate the effectiveness of the method using a range of numerical experiments with a selection of functions and matrices. We primarily focus on the sign function arising in the overlap formalism of lattice QCD. 
\end{abstract}

% REQUIRED

% REQUIRED
\overfullrule=0pt %Had to add this to avoid black boxes at the end
\section{Introduction}
Many applications in scientific computing require the evaluation of
$f(A) \vb$ where 
$ A \in \mathbb{C}^{n \times n}$ is some matrix, $\vb \in \mathbb{C}^{n}$ is a vector, and $f$ is some complex valued matrix function such that $f(A)$ is defined. For example, the matrix exponential  function $f(A) = \mbox{exp}(A)$ arises in the solution of ordinary differential equations \cite{hochbruck1997krylov, hochbruck2010exponential}. The matrix logarithm $f(A) = \mbox{log}(A)$  arises in Markov model analysis \cite{singer1976representation}, and the matrix sign function $f(A) = \mbox{sign}(A)$  arises in lattice QCD simulations with overlap fermions \cite{knechtli2017lattice, frommer2000numerical}. Perhaps the most well known example is the solution of a linear system where we have $f(A) = A^{-1}$.

In High Performance Computing (HPC) applications the matrix $A$ is large and sparse, and sometimes not available explicitly. In this setting, ``matrix-free'' Krylov subspace methods are the only efficient means of computing an approximation to $f(A)\vb$ by projecting the problem onto a $j$-dimensional Krylov subspace $\mathcal{V}_{j}$  ($j \ll n$), and treating a smaller representation of the problem on this subspace via a direct method.
 
In this paper we develop a Krylov subspace recycling method for treating a sequence of a matrix functions applications of the form
\begin{equation}\label{eq:fAibi}
    f (A^{(i)}) \vb^{(i)}, \enspace i=1,2,\ldots,N,
\end{equation}
where $A^{(i)}, i = 1,2, \ldots $ is a sequence of matrices, and $\vb^{(i)}, i = 1,2,\ldots $ is a sequence of vectors.
Krylov subspace recycling aims to accelerate the convergence of a standard Krylov subspace method by augmenting the Krylov subspace $\mathcal{V}_{j}$ with an additional $k$ dimensional subspace $\mathcal{U}$ known as the \textit{augmentation subspace}. The space $\mathcal{U}$ is generally chosen to contain useful information to aid in accelerating the convergence of the method, and is often taken to be approximations to the $k$ eigenvectors corresponding to the $k$ eigenvalues closest to a singularity of the function $f$. The method in this paper is the first of its kind proposed to extend subspace recycling to a sequence of matrix function applications where the matrices in the sequence are slowly changing and/or the vectors $\vb^{(i)}$ change as the sequence progresses.

We develop an augmented Krylov subspace approximation to $f(A) \vb$ by deriving an augmented Krylov subspace approximation to the shifted linear systems appearing in its integral representation. Although augmentation of Krylov subspaces has already been proposed for a single matrix function application (see eg. \cite{bloch2007iterative, eiermann2011deflated}), the purpose of this paper is to derive a method which allows for subspace recycling between a sequence of $N$ matrix function applications of the form \eqref{eq:fAibi}.  It should be noted that unlike, e.g., the deflated-restarts approach \cite{eiermann2011deflated}, there is no restriction on our choice of $\mathcal{U}$, which allows us to reuse and refine $\mathcal{U}$ as $\vb^{(i)}$ and possibly $A^{(i)}$ are changing as the sequence progresses. 

The layout of this paper is as follows. In Section \ref{background} we give some background on the literature for augmented Krylov subspace methods for matrix functions. We also discuss the GCRO-DR algorithm for linear systems \cite{parks2006recycling} which serves as an inspiration for the method we derive in this work. In Section \ref{Arnoldiapproximation} we discuss the standard Arnoldi approximation to $f(A) \vb$ and in particular focus on its derivation using the Cauchy integral representation  of the matrix function. In Section \ref{signfunctioninLQCD} we discuss computational challenges arising in the overlap formalism of lattice QCD which has served as our main motivation for this work. In Section \ref{solutionofshiftedsystems} we introduce an augmented Krylov subspace framework for solving the shifted linear systems arising in the integral representations of $f(A) \vb$. In Section \ref{r(FOM)2} we show how this framework for shifted systems can be used to derive an augmented Krylov subspace approximation for $f(A)\vb$. We present three different versions of a FOM type method, leading to different implementations. A brief discussion on Stieltjes functions can be found in Section \ref{sec:stieltjes}. In Section \ref{harmritz} we derive a Ritz procedure for updating the recycling subspace between problems.  In Section \ref{numericalexperiments} we present numerical experiments and test the three versions of the method as both an augmented method for a single problem, and a recycling method for a sequence of problems. Brief conclusions are given in Section \ref{conclusions}.  

Throughout this paper, lower-case bold letters denote vectors, while upper case letters are used to denote matrices. Matrices are written in bold if they are explicitly defined as a concatenation of other matrices.

\section{Background}\label{background}
In recent years there has been an interest in improving computational aspects of the Arnoldi approximation for matrix functions in order to keep up with application demands and practical computing limitations. A procedure for implementing restarts into the Arnoldi approximation by computing an error function using divided differences was first introduced in \cite{eiermann2006restarted}. In \cite{frommer2014efficient} it was shown that expressing this error function as an integral allows for a more numerically stable restart procedure based on quadrature. A convergence analysis was later provided in \cite{FrGuSch14}.  A framework for treating the action of a matrix function on multiple right hand sides using \textit{block} Krylov subspace methods was derived in works such as \cite{frommer2017block, frommer2020block}.

The augmented Krylov subspace methods proposed  in the literature aim at treating problems where the matrix and vector remain fixed. One of the earlier works \cite{bloch2007iterative} has developed a deflation technique which precomputes Schur vectors from the original matrix and runs a single Arnoldi cycle which orthogonalizes the Arnoldi vectors against these Schur vectors. Similarly, a procedure for performing deflation between restarts of the Arnoldi process was introduced in \cite{eiermann2011deflated}, making it possible to avoid the computational expense of precomputing a deflation subspace from the original matrix. The methods in these papers use an \textit{Arnoldi-like} decomposition \cite{eiermann2011deflated}, arising from the fact that the augmented subspace is still a Krylov subspace. This means that a  generalization of the Arnoldi approximation could be used \emph{in that setting}. 

This is similar to ideas presented in the literature for Krylov subspace methods for the solution of a linear system. Morgan \cite{morgan2002gmres} showed that  when GMRES is deflated with harmonic Ritz vectors between restarts, the augmented Krylov subspace is still a Krylov subspace but with a different starting vector, allowing  the approximation to be computed in a manner similar to standard GMRES. The GMRES-DR algorithm is thus restricted to choose the augmentation subspace to be harmonic Ritz vectors, and in addition, can only be used to treat a single linear system.  The Generalized Conjugate Residual Orthogonalization method with Deflated Restarting (GCRO-DR) algorithm
  \cite{parks2006recycling} overcomes these difficulties by
combining ideas from GMRES-DR and GCRO \cite{de1996nested} to allow for an arbitrary choice of $\mathcal{U}$, as well as allowing for recycling between a sequence of changing linear systems. We refer the reader to \cite{ parks2006recycling,morgan2002gmres} for further details on these methods.

In this paper we extend the ideas in \cite{eiermann2011deflated} to methods that allow a recycling strategy to be employed for matrix function applications of the form (\ref{eq:fAibi}) using an \emph{arbitrary subspace} $\mathcal{U}$. The method is derived from a recycled shifted Full Orthogonalization Method (FOM) for the solution to a shifted linear system.

\section{The Arnoldi approximation to $f(A) \vb$} \label{Arnoldiapproximation}

We shortly summarize the essentials on the theoretical justification and the computation of the well-known Arnoldi approximation to $f(A) \vb$. We take inspiration from this approach when developing the augmented Krylov subspace approximation.

Any matrix function $f(A)$ is identical  to a polynomial $p(A)$, where $p$ interpolates $f$ at the eigenvalues of $A$ in the Hermite sense; see \cite{Higham2008}. Thus the Krylov subspace

\[
    \mathcal{K}_{j}( A , \vb) = \{ q(A) \vb : \mbox{ $q$ polynomial of degree $< j$}\}
\]
is the natural subspace from which to choose an appropriate approximation
to $f(A)\vb$. In the case of a linear system $f(A) = A^{-1}$, a Galerkin condition can be used to define an approximation $\vx_j \in \mathcal{K}_{j}( A , \vb)$ via the orthogonality condition
\begin{equation} \label{FOM:eq}
\vb - A\vx_j \perp \mathcal{K}_{j}( A , \vb).
\end{equation}

Usually, this {\em FOM approximation} \cite{saad2003iterative} is obtained from the 
Arnoldi process which computes a nested orthonormal basis for $\mathcal{K}_{j}( A , \vb)$. Arranging the 
basis vectors as the columns of a matrix $V_{j} \in \mathbb{C}^{n \times j}$, the Arnoldi process can be 
summarized via the Arnoldi relation  
\begin{equation}\label{Arnoldi}
    A V_{j} = V_{j+1} \overline{H}_{j} = V_{j} H_{j} + h_{j+1,j} \vv_{j+1}\ve_{j}^{T},
\end{equation}
where $\overline{H}_{j} \in \mathbb{C}^{(j+1) \times j}$ is an upper Hessenberg matrix and $H_{j} \in 
\mathbb{C}^{j \times j}$ arises from $\overline{H}_{j}$ by deleting the last row. The FOM iterate $\vx_j$ satisfying \eqref{FOM:eq} can then be computed as $\vx_j = \|\vb\| V_j H_j^{-1} \ve_1$, where $\ve_1$ is the first canonical unit vector in $\mathbb{C}^{j}$. 

A crucial observation for matrix functions is the fact that Krylov subspaces are invariant under shifts of the matrix by scalar multiples of the identity
\[
\mathcal{K}_{j}( A , \vb)  = \mathcal{K}_{j}( \sigma I  - A , \vb)
\]
for all $\sigma \in \mathbb{C}$, and that the FOM approximations for $ (\sigma I - A)^{-1}\vb$ are given by
\[
\vx_j(\sigma) = \|\vb\| V_j (\sigma I-H_j)^{-1} \ve_1,
\]
with $V_j$ and $H_j$ built from the Arnoldi process for $A$; see e.g. \cite{simoncini2003restarted}.

For functions $f$ which are analytic in an appropriate region, such shifted systems arise when expressing $f(A)$ using Cauchy's integral formula as given in the following definition; see \cite{Higham2008}. 

\begin{definition}
For the matrix $A \in \mathbb{C}^{n \times n}$ with spectrum $\Lambda(A)$, if $f$ is analytic in a region containing the closed contour $\Gamma$, which in turn contains $\Lambda(A)$ in its interior, $f(A)$ is given as \normalfont 
\begin{equation} \label{integral_representation:eq}
    f(A) = \frac{1}{2 \pi i} \int_{\Gamma} f(\sigma) (\sigma I - A)^{-1} \d \sigma.
\end{equation}
\end{definition}

We may thus write \begin{equation}\label{eq:caughy_integral_rep}
    f(A) \vb = \frac{1}{2 \pi i} \int_{\Gamma} f(\sigma) (\sigma I - A)^{-1} \vb \hspace{0.1cm} \d \sigma =  \frac{1}{2 \pi i} \int_{\Gamma} f(\sigma) \vx(\sigma) \hspace{0.1cm} \d \sigma
\end{equation}
where $\vx(\sigma)$ is the solution of the shifted linear system
\begin{equation}\label{matfunshiftedsystem}
    (\sigma I - A) \vx(\sigma) = \vb.
\end{equation}
Using the FOM approximations $\vx_j(\sigma)$ to the solution $\vx(\sigma)$ for each $\sigma$, we then get the approximation $\vf_{j} \approx f(A) \vb$ defined by
\begin{equation} \label{Arnoldi_Cauchy:eq}
    \vf_{j} := \frac{\| \vb \|}{2 \pi i} \int_{\Gamma} f(\sigma) V_{j} (\sigma I - H_{j})^{-1} \ve_{1} \hspace{0.1cm} \d \sigma     = \| \vb \| V_{j} f(H_{j}) \ve_{1},
\end{equation}
where the last equality only holds if the spectrum of $H_j$ lies inside the contour $\Gamma$. Approximation \eqref{Arnoldi_Cauchy:eq} is known
as the \textit{$j$th Arnoldi approximation} to $f(A)\vb$, and is defined if $f(H_j)$ is defined, i.e.\ when $f$ is defined on the spectrum of $H_j$ in the sense of \cite{Higham2008}. Note that we can express $f(H_j)$ via the Cauchy integral formula using any admissible contour which encloses the spectrum of $H_j$, but not necessarily that of $A$.

\section{The sign function in lattice QCD}
\label{signfunctioninLQCD} 

This work has been motivated by computational challenges arising in lattice QCD simulations \cite{joo2019status}.
The sign function of a large, non-Hermitian matrix appears in the overlap operator in lattice QCD.
In QCD, the Wilson-Dirac operator $D$ arises as a discretization of the Dirac operator on a finite four-dimensional Cartesian lattice. The Wilson-Dirac operator acts on discrete spinor fields which have twelve components per lattice point, corresponding to all possible combinations of three color and four spin indices \cite{gattringer2009quantum}.
More precisely, the Wilson-Dirac operator is defined via a configuration of {\em gauge links} $U_\nu(\ell)$ for each lattice site $\ell$ and each direction $\nu \in \{1,2,3,4\}$ in Euclidean space-time. The gauge links $U_\nu(\ell)$ belong to the $SU(3)$ symmetry group. In the Wilson-Dirac operator, the twelve spinor components at lattice site $\ell$ then couple to the spinors at the (periodic) nearest neighbours in all four dimensions. The coupling coefficient for a neighbour in positive direction $\nu$ is given as the tensor product of $U_\nu(\ell)$ and the $4 \times 4$ matrix $I+\gamma_\nu$, and similarly for the coupling in negative direction. The matrices $\gamma_\nu$ are representations of a Clifford algebra and do not depend on $\ell$. 

The {\em overlap Dirac operator}, $D_{ov}(\mu)$ preserves chiral symmetry, an important physical property, on the lattice while other discretizations such as, e.g., $D$ do not. 
To be specific, the overlap Dirac operator takes the form~\cite{bloch2007iterative,hernandez1999locality,Neuberger1998}
\begin{equation*}\label{eq:overlap_operator}
    D_{ovl}(\mu) = I + \rho \Gamma_{5} \sign{( \underbrace{\Gamma_{5} D(\kappa,\mu) )}_{=: Q(\kappa,\mu)}}.
\end{equation*}
Here, $\Gamma_5$ is a simple diagonal matrix which acts as the identity on spinor components belonging to spins 1 and 2 and as the negative identity on  those belonging to spins 3 and 4, and $\rho \in (0,1)$ is a mass parameter, typically close to $1$.
In the argument of the sign function, $D(\kappa,\mu)$ is the Wilson-Dirac operator with $\kappa = \frac{4}{3}\kappa_{c}$ where $\kappa_{c}$ is the ``critical value'', i.e.\ the smallest value of $\kappa$ for which the smallest real part of an eigenvalue in $D(\kappa,\mu)$ is zero. The parameter
$\mu$ designates the chemical potential, and it is the presence of $\mu \neq 0$ that makes $Q(\kappa,\mu)$ non-Hermitian; see \cite{bloch2007iterative}. For notational simplicity, we abbreviate $Q(\kappa,\mu)$ as $Q_\mu$ and $D(\kappa,\mu)$ as $D_\mu$ from now on. Our interest in this paper is on computing $\mbox{sign}(Q_{\mu})\vb$; in \cite{bloch2007iterative} different methods are explored for this problem.

% The overlap formalism \cite{neuberger1998exactly} introduces the overlap Dirac operator
% \[
%     D_{ovl} = \rho I + \Gamma_{5} \mbox{sign}(Q).
% \]
% Here, $\rho$ is a mass parameter, $Q = \Gamma_{5} D$, wherein the Dirac-Wilson matrix $D$ represents a periodic nearest neighbour coupling on a 4 dimensional space-time lattice with 12 unknowns per lattice site (corresponding to all possible combinations of four spins and three colors). The matrix $\Gamma_{5}$ acts identically on the components at each lattice site by flipping half of the spins. As a consequence, $Q = \Gamma_5 D$ is Hermitian, $Q^* = Q$. \gr{We consider here an extension of this where the Dirac-Wilson operator has a \textit{chemical potential} $\mu$ \cite{gattringer2009quantum,bloch2007iterative}:
% \[
%     (D_{\mu})_{nm} = \delta_{nm} - \kappa \sum_{j=1}^{3}  .
% \]
% %, and represent by $D_{\mu}$ the Dirac-Wilson operator in the presence of $\mu$,
% The matrix $Q$ becomes now $Q_{\mu} := \Gamma_{5} D_{\mu}$, for which $Q^{*}_{\mu} = Q_{-\mu}$. Our interest in this paper is on computing $\mbox{sign}(Q_{\mu})\vb$; in \cite{bloch2007iterative} different methods are explored for this problem.}

The overlap formalism has significant importance for %computing 
certain types of observables 
to be 
computed in lattice QCD simulations. These simulations are thus dependent on overcoming the algorithmic complexity challenges associated with the overlap discretization.

Lattice QCD computations with the overlap formalism are dominated by computing solutions to a sequence of $N$ linear systems of the form
\begin{equation}\label{eq:overlap_linear_system}
  D_{ovl}^{(i)} \vx^{(i)} = \vb^{(i)}, \hspace{1cm} i=1,\ldots,N.
\end{equation}
The computation of quark propagators \cite{joo2019status} is a special case wherein the matrices $D_{ovl}^{(i)}$ remain fixed and the vectors $\vb^{(i)}$ are changing from one system to the next. In the initial phase of generating QCD gauge fields in the Markov Chain Monte Carlo process, the configuration of gauge links and thus the matrix $D_{ovl}^{(i)}$ is slowly changing for every system \cite{cundy2009numerical}.

The lattices in lattice QCD are large and thus the Wilson-Dirac matrix $D_\mu$ as well as $Q_\mu$ are large and sparse. Due to the sign function, the overlap operator represents a large and full matrix and, as such, cannot be computed explicitly. Therefore,  computing an approximation of the solution for the systems in \eqref{eq:overlap_linear_system} can only be achieved via an iterative method, usually a preconditioned Krylov subspace method. %The Arnoldi algorithm can be used to build a basis for the Krylov subspace $\mathcal{K}_{j}( D_{ovl}^{(i)} , \vr_{0}^{(i)}  )$, where $\vr_{0}^{(i)}$ is the initial residual corresponding to an initial approximation $\vx_{0}^{(i)}$ for system $i$. 
Each iteration %of the Arnoldi algorithm 
then requires to compute the application of $D^{(i)}_{ovl}$ to a new vector. Thus, for each system $i$ in \eqref{eq:overlap_linear_system} this requires the evaluation of $\mbox{sign}(Q_{\mu}^{(i)}) \vv^{(\ell)}$ for a fixed matrix $Q_{\mu}^{(i)}$ sequentially for a series of vectors $\vv^{(\ell)}$, $\ell = 1, \ldots , j$, where $j$ is the number of iterations.  After each system solve, the matrix $Q_{\mu}^{(i)}$  either remains fixed, or changes to a new matrix $Q_{\mu}^{(i+1)}$, and the process repeats.

The application of the sign function of each matrix on a set of vectors is a major computational bottleneck in the overlap formalism, and it requires an additional inner Krylov subspace method to compute the action of $\sign(Q_{\mu}^{(i)})$ within each single iterative step in the solution process for each system in the sequence \eqref{eq:overlap_linear_system}. There has been much interest in developing new methods to address the computational challenges associated to the overlap formalism. We refer the reader to works such as \cite{frommer2000numerical}, as well as the series of publications \cite{cundy2009numerical,arnold2005numerical,cundy2005numerical,van2002numerical}. We also refer the reader to \cite{BraFroKaLeRoStr2015} where a multigrid preconditioning for the overlap operator was considered.
In previous work \cite{van2002numerical}, exact deflation was used as a means to accelerate the inner iteration using a fixed number of the smallest eigenpairs of each $Q_{\mu}^{(i)}$ to high accuracy before starting the computation of the sign function.   
We propose an alternative using Krylov subspace recycling instead, which interweaves the computation of increasingly accurate eigenmodes with the computation of the sign function itself. 

\section{A derivation of an augmented Krylov subspace approximation to shifted linear systems% of the form $(\sigma I - A) \vx(\sigma) = \vb$
}\label{solutionofshiftedsystems}
We derive an augmented Krylov subspace approximation for $f(A) \vb$ in a similar manner as for the Arnoldi approximation presented in Section \ref{Arnoldiapproximation} for the non-augmented case. We first derive an augmented Krylov subspace approximation for the family of shifted linear systems (\ref{matfunshiftedsystem}) appearing in the integral representation (\ref{eq:caughy_integral_rep}). We make use of a residual projection framework which describes almost all augmented Krylov subspace methods for standard non-shifted linear systems. The framework was originally proposed in terms of typical residual constraints in works such as \cite{gaul2014recycling, gaul2013framework, gutknecht2012spectral,  gutknecht2014deflated}, but a recent survey \cite{Soodhalter2020ASO} extends the framework to the case of arbitrary search and constraint spaces. The work in this section is based on the most general form of the framework as presented in this survey.

The  framework describes augmented Krylov subspace methods broadly in two main steps. The first applies a standard Krylov subspace method to an appropriately projected problem, and the second performs an additional projection into the augmentation subspace.
We present an extension of this framework to the case of a family of shifted linear systems (\ref{matfunshiftedsystem}). We assume that the initial approximations for all shifts are zero and that the approximation $\vx_j(\sigma)$ for every shift $\sigma$ is of the form
\[
\vx_j(\sigma) =  \vs_j(\sigma) + \vt_j(\sigma) 
\]
with $ \vt_{j}(\sigma)$ taken from the same $j$-dimensional Krylov subspace $\mathcal{V}_{j} $ for all $\sigma$ and $\vs_j(\sigma)$ from a fixed $k$- dimensional augmentation subspace $\mathcal{U}$ which does not depend on the iteration step $j$. We use the $j$ subscript on $\vs$ since, as we show, the contribution from the augmentation subspace also depends on $j$ even though $\mathcal{U}$ is fixed. With the columns of $U$ spanning $\mathcal{U}$ and those of $V_j$ spanning $\mathcal{V}_j$ the approximation  $\vx_{j}(\sigma)$ at step $j$ can be expressed as 
\[
    \vx_{j}(\sigma) =  \vs_{j}(\sigma) + \vt_{j}(\sigma)  =   U \vz_{j}(\sigma) + V_{j} \vy_{j}(\sigma)
\]
with $\vz_{j}(\sigma) \in \mathbb{C}^{k}$ and $\vy_{j}(\sigma) \in \mathbb{C}^{j}$ to be determined. This is done by imposing that for all shifts $\sigma$ the residual be orthogonal to the sum of some other $j$-dimensional Krylov subspace $\widetilde{\mathcal{V}}_{j}$, spanned by the columns of $\widetilde{V}_j \in \mathbb{C}^{n\times j}$,  and some other augmentation subspace $\widetilde{\mathcal{U}}$, spanned by the columns of $\widetilde{U} \in \mathbb{C}^{n\times k}$,
\[
    \vr_{j}(\sigma) = \vb - (\sigma I - A)(  U \vz_{j}(\sigma) + V_{j} \vy_{j}(\sigma)   ) \perp  \widetilde{\mathcal{U}} + \widetilde{\mathcal{V}}_{j}.
\]
 yielding the following linear system for $\vz_{j}(\sigma)$ and $\vy_{j}(\sigma)$:
\begin{equation}\label{eq:mat_eqn}
\begin{bmatrix} \widetilde{U}^{*}(\sigma I - A) U &    \widetilde{U}^{*}(\sigma I - A) V_{j} \\
   \widetilde{V}^{*}_{j}(\sigma I - A) U &  \widetilde{V}^{*}_{j}(\sigma I - A) V_{j}
  \end{bmatrix} \begin{bmatrix}
        \vz_{j}(\sigma) \\
        \vy_{j}(\sigma) 
    \end{bmatrix} = \begin{bmatrix}
        \widetilde{U}^{*} \vb \\
       \widetilde{V}_{j}^{*} \vb
  \end{bmatrix}.
\end{equation}

Eliminating the bottom left block of (\ref{eq:mat_eqn}) we can express $\vy_j(\sigma)$ as the solution of 
\begin{equation} \label{yj:eq}
    \widetilde{V}_{j}^{*} (I - P_{\sigma})(\sigma I - A) V_{j} \vy_{j}(\sigma) =  \widetilde{V}_{j}^{*} (I - P_{\sigma}) \vb,
\end{equation}
with the projector $P_{\sigma} := (\sigma I - A) U (\widetilde{U}^{*}(\sigma I - A) U)^{-1} \widetilde{U}^{*}$.
The coefficients $\vz_j(\sigma)$ for the contribution from the augmenting subspace are then obtained from $\vy_j(\sigma)$ as
\begin{equation} \label{zj:eq}
    \vz_{j}(\sigma) = (\widetilde{U}^{*} (\sigma I - A) U)^{-1}\widetilde{U}^{*} \vb -  (\widetilde{U}^{*} (\sigma I - A) U)^{-1}\widetilde{U}^{*} (\sigma I - A) V_{j} \vy_{j}(\sigma).
\end{equation}

From \eqref{yj:eq} we see that computing $\vy_{j}(\sigma)$ is equivalent to applying the following projection method.
\medskip

\begin{mdframed}
   Find $\vt_{j}(\sigma) \in \mathcal{V}_{j}$ as an approximate solution corresponding to shift $\sigma$ for the projected and shifted linear system
   \begin{equation} \label{projectedproblem}
    (I - P_{\sigma})   ( \sigma I - A) \vx(\sigma) = (I - P_{\sigma}) \vb
   \end{equation}
    such that $\vr_{j}(\sigma) = (I - P_{\sigma} ) (\vb - (\sigma I - A) \vt_{j}(\sigma)) \perp \widetilde{\mathcal{V}}_{j} $.
\end{mdframed}
\medskip % \vspace{0.5cm}

The linear system \eqref{projectedproblem} is clearly singular, due to the left multiplication of the projector $(I - P_{\sigma})$, and thus the solution is non-unique. However, after computing $\vy_{j}(\sigma)$, a second projection step in computing $\vz_{j}(\sigma)$ then recovers the unique solution to the original linear system $(\sigma I - A) \vx(\sigma) = \vb$. Additionally, the subspaces $\widetilde{\mathcal{V}}_{j}$ and $\mathcal{U}$ should be chosen such that the projector $P_{\sigma}$ is well defined.

Although a natural choice for $\mathcal{V}_{j}$ is the Krylov subspace corresponding to the projected problem 
\[
\mathcal{K}_{j}((I - P_{\sigma})(\sigma I - A), (I - P_{\sigma}) \vb).
\] 
(see e.g, \cite{Soodhalter2020ASO}) this Krylov subspace is built from a projected operator which depends on $\sigma$, and thus shift invariance no longer holds, and a practical implementation would require generation of a separate Krylov subspace for each $\sigma$. This has been shown to be one of the main difficulties associated with developing a subspace recycling method for the solution of shifted systems \cite{soodhalter2016block,soodhalter2014krylov}.

We propose to overcome this difficulty using the observation made in \cite{burke2022augmented}, that the projected problem (\ref{projectedproblem}) has an unprojected equivalent given by:
\medskip %\vspace{0.5cm}

\begin{mdframed}
   Find $\vt_{j}(\sigma) \in \mathcal{V}_{j}$  as an approximate solution corresponding to shift $\sigma$ for the shifted linear system
   \begin{equation}\label{eq:unprojected_problem}
     ( \sigma I - A) \vx(\sigma) =  \vb
   \end{equation}
    such that $\vr_{j}(\sigma) =  \vb - (\sigma I - A) \vt_{j}(\sigma) \perp (I - P_{\sigma})^{*} \widetilde{\mathcal{V}}_{j} $.
\end{mdframed}
\medskip
%\vspace{0.5cm}
We can thus solve the unprojected problem \eqref{eq:unprojected_problem} using the unprojected, shift-invariant Krylov subspace $ \mathcal{K}_{j}(A,\vb)$, and deal with the projector $(I - P_{\sigma})$ in the computation for $\vy_{j}(\sigma)$ in \eqref{yj:eq}.
Once we have computed $\vy_{j}(\sigma)$ we can incorporate $\vz_{j}(\sigma)$ from \eqref{zj:eq} implicitly into the full solution approximation to get
\begin{equation}\label{eq:fullshiftedsolution}
    \vx_{j}(\sigma) =  V_{j} \vy_{j}(\sigma) + U(\widetilde{U}^{*} (\sigma I - A) U)^{-1}\widetilde{U}^{*} (\vb -  (\sigma I - A) V_{j} \vy_{j}(\sigma)),
\end{equation}
which can also be rewritten as
\begin{equation}\label{eq:fullshiftedsolutioncompact}
    \vx_{j}(\sigma) = \vV_{j} \begin{bmatrix}
         \widetilde{\vV}^{*}_{j}(\sigma I - A) \vV_{j}
    \end{bmatrix}^{-1} \widetilde{\vV}^{*}_{j} \vb
\end{equation}
with  $ \vV_{j} = \begin{bmatrix} U & V_{j} \end{bmatrix} \in \mathbb{C}^{N \times (k+j)}$ and $\widetilde{\vV}_{j} = \begin{bmatrix} \widetilde{U} & \widetilde{V}_{j} \end{bmatrix} \in \mathbb{C}^{N \times (k+j)}$.

Approximations (\ref{eq:fullshiftedsolution})
and (\ref{eq:fullshiftedsolutioncompact}) are thus two mathematically equivalent approximations to $\vx(\sigma)$, and can be used to derive any new augmented unprojected Krylov subspace method for shifted systems (\ref{matfunshiftedsystem}) using appropriate choices of the subspaces $\widetilde{\mathcal{V}}_{j}$ and $\widetilde{\mathcal{U}}$. For example, the choice $\widetilde{\mathcal{V}}_{j} = A \mathcal{K}_{j}(A,\vb)$ and $\widetilde{\mathcal{U}} = A \mathcal{U}$ leads to an unprojected augmented GMRES-type method for shifted systems, while the choice $\widetilde{\mathcal{V}}_{j} = \mathcal{K}(A,\vb)$ and $\widetilde{\mathcal{U}} = \mathcal{U}$ leads to an unprojected augmented FOM-type method which can be viewed as an extension of the unprojected augmented FOM method derived in \cite{burke2022augmented} to the case of shifted systems of the form (\ref{matfunshiftedsystem}). Since we are dealing with matrix functions, where the standard Arnoldi approximation is of FOM-type, we now focus on FOM type methods and make these choices. 

\section{Recycled FOM (quad) for functions of matrices}\label{r(FOM)2}

Taking the results from the previous section and using the integral representation \eqref{eq:caughy_integral_rep}
we explicitly define the augmented unprojected FOM approximation of $f(A)\vb$ with subspaces  $\widetilde{\mathcal{V}}_{j} = \mathcal{V}_{j} = \mathcal{K}_{j}(A,\vb)$ and $\widetilde{\mathcal{U}} = \mathcal{U}$ as
\begin{align}\label{integralform}
    f(A) \vb \approx \frac{1}{2 \pi i} \int_{\Gamma} f(\sigma) \vx_{j}(\sigma) \hspace{0.1cm} \d \sigma,
\end{align}
where $\vx_j(\sigma)$ is the shifted augmented FOM approximation to the solution of the shifted system $(\sigma I-A)\vx = \vb$ given in (\ref{eq:caughy_integral_rep}). The use case that we have in mind is that we must compute a sequence of expressions of the form $f(A)\vb$ with varying $\vb$ and/or $A$ and that we use information acquired from the computation of one expression to obtain good choices for the augmenting space $\mathcal{U}$. Due to the need to evaluate the integral expression using quadrature, we refer to our new method as \textit{recycled FOM (quad)}.

Depending on how we treat \eqref{integralform} computationally, and in particular on how we use quadrature rules on the different integrals  occurring in equivalent representations of \eqref{integralform}, we obtain three versions of the rFOM (quad) method that we discuss below. We use the notation $n_{quad}$ to denote the number of quadrature points used, and $z_\ell$ and $\omega_\ell$, $\ell = 1,\ldots,n_{\quadp}$ to denote the nodes and the weights, respectively, of a quadrature rule for the contour $\Gamma$ which yields the approximation
\[
 \frac{1}{2\pi i} \int_{\Gamma} f(\sigma) \vx_{j}(\sigma) \hspace{0.1cm} \d \sigma \approx \sum_{\ell=1}^{n_{\quadp}} \omega_\ell f(z_\ell) \vx_{j}(z_{\ell}).
\]

Approximating $f(A) \vb$ in its integral form via quadrature has been discussed in works such as  \cite{boricci2003qcd, hale2008computing}. In addition, numerical quadrature also plays a role in numerically stable restart procedures in the Arnoldi approximation \cite{frommer2014efficient}. We stress 
%again
that our three versions differ because they use numerical integration for different integrals.

Due to the nature of the expressions we derive inside the Cauchy integral, it is not possible to modify the contour to just contain the spectrum of a smaller matrix as with the standard Arnoldi approximation; see Section~\ref{Arnoldiapproximation}.  
As we see it may be necessary to enlarge the contour to enclose additional points depending on the choice of $\mathcal{U}$.  This is a potential limitation of our method worth considering, since it requires \textit{a priori} information on the spectrum of each matrix in order to do the quadrature approximation. We note that in our numerical experiments we constructed a contour which we knew only enclosed the spectrum of $H_{j}$, and this did not pose any issues. 

The common feature of all three versions is that we perform the Arnoldi process to compute the orthogonal basis of $\mathcal{V}_j = \mathcal{K}_{j}(A,\vb)$. This requires $j$ matrix-vector multiplications and $\mathcal{O}(nj^2)$ arithmetic work for the orthogonalization. The three versions differ in 
the integrals that are evaluated numerically, and we discuss the cost of each version below. 
For further implementation details, we refer the reader to the MATLAB code\footnote{Code available at \url{https://github.com/burkel8/MatrixFunctionRecycling }}.

\subsection{A quadrature based implementation of rFOM (quad)}
We derive a quadrature based implementation of rFOM (quad) by substitution of (\ref{eq:fullshiftedsolution}) into \eqref{integralform} with $\widetilde{\mathcal{U}} = \mathcal{U}$ and $\widetilde{\mathcal{V}}_{j} = \mathcal{K}_{j}(A, \vb)$. 
The solution approximation (\ref{eq:fullshiftedsolution}) first requires solving equation (\ref{yj:eq}) for $\vy_{j}(\sigma)$, which can be done cheaply after simplifying (\ref{yj:eq}) to
\begin{equation} \label{eq:version1_y}
    (V_{j}^{*}(I-P_{\sigma}) V_{j+1}(\sigma \overline{I} - \overline{H}_{j})) \vy_{j}(\sigma) = V^{*}_{j} (I-P_{\sigma}) \vb,
\end{equation}
where $C = AU$,  $P_{\sigma} = (\sigma U - C) (\sigma U^{*} U - U^{*} C)^{-1} U^{*}$ and $\overline{I} \in \mathbb{R}^{(j+1) \times j}$ is a $j \times j$ identity matrix padded with an additional row of zeros.
The expression (\ref{eq:fullshiftedsolution}) is then given as
\begin{equation}\label{eq:shifted_sol_1}
    \vx_{j}(\sigma) =  V_{j} \vy_{j}(\sigma) + U(\sigma U^{*} U - U^{*} C)^{-1} U^{*}( \vb -  V_{j+1}(\sigma \overline{I} - \overline{H}_{j})\vy_{j}(\sigma),
\end{equation}
and, for numerical stability, should be computed after orthonormalizing the columns of $U$ (such that $U^*U = I$).

Using (\ref{eq:shifted_sol_1}), the first version of rFOM (quad) computes the approximation $\widetilde{\vf}_{1}$ to $f(A)\vb$ as
\begin{align*}
    \widetilde{\vf}_{1} := &
      V_{j} \sum_{\ell=1}^{n_{\quadp}} \omega_\ell f(z_{\ell}) \vy_{j}(z_{\ell}) \\
      &+   U \sum_{\ell=1}^{n_{\quadp}} \omega_\ell f(z_{\ell}) (  z_{\ell} I - U^{*} C )^{-1} ( U^{*} \vb - U^{*} V_{j+1}(z_{\ell} \overline{I} - \overline{H}_{j})\vy_{j}(z_{\ell})).
\end{align*}

For computational efficiency, it is appropriate to compute the coefficient matrix $V_{j}^{*}(I-P_{\sigma}) V_{j+1}(\sigma \overline{I} - \overline{H}_{j})$ for $\vy_j$ in \eqref{eq:version1_y} as 
\begin{equation} \label{eq:matrix-product}
 \sigma I - H_{j}  - K_{j}(\sigma)L(\sigma)M(\sigma \overline{I} - \overline{H}_j) \
\end{equation}
with
\[
K_{j}(\sigma) = \sigma V_j^*U-V_j^*C, \enspace 
L(\sigma) = (\sigma I -U^*C)^{-1}, \enspace \mbox{and} \enspace 
M_{j} = U^*V_{j+1},  
\]
which avoids redundant computations. 

This algorithm requires non-singularity of the matrix $U^{*}(\sigma I - A) U$, which from the following proposition, is shown to be guaranteed provided that the contour of integration encloses the field of values of $A$, defined as 
\[
\mathcal{F}(A) := \{ \vt^{*} A \vt, \vt \in \mathbb{C}^n,  \| \vt \|_{2} = 1 \},
\] 
which is a convex superset of the spectrum of $A$; see \cite{HornJohnson}.

\begin{proposition}
If $U$ has orthonornal columns and the contour $\Gamma$ encloses $ \mathcal{F}(A)$, then the matrix $U^{*}(\sigma I - A) U$ is non-singular for all $\sigma \in \Gamma$.
\end{proposition}

\begin{proof}
We have $U^*(\sigma I-A)U =  \sigma I - U^*AU$ and thus
\[
\mathcal{F}(U^*(\sigma I-A)U ) = \sigma - \mathcal{F}(U^*AU) \subseteq \sigma - \mathcal{F}(A).
\]
Since the contour $\Gamma$ encloses $\mathcal{F}(A)$, we have $0 \not \in \sigma- \mathcal{F}(A)$ for all $\sigma \in \Gamma$, and thus $0 \not \in \mathcal{F}(U^*(\sigma I-A)U$), implying that $U^*(\sigma I-A)U$ is non-singular.
\end{proof}

\begin{algorithm}
\caption{rFOM (quad) version ~1} 
\label{version1:alg}
\begin{algorithmic}[1]
\STATE{\textbf{Input:} $A \in \mathbb{C}^{n \times n}$,  $\vb \in \mathbb{C}^{n}$, scalar function $f$, $U \in \mathbb{C}^{n \times k} $ whose columns span the augmenting recycled subspace $\mathcal{U}$,  $C = A U$, Arnoldi cycle length $j$ }
\STATE{Determine contour containing the spectrum of  $A$ and determine quadrature rule with nodes $z_\ell$ and weights $\omega_\ell, \ell = 1,\ldots,n_{\quadp}$}
\STATE{Set $\vt_{1} = \vnull  \in \mathbb{C}^{j}$, $\vt_{2} = \vnull \in  \mathbb{C}^{k}$} 
\FOR{$\ell = 1, \ldots,  n_{\quadp}$}
\STATE{$\mu \leftarrow \omega_\ell f(z_\ell)$} 
\STATE{Solve $( z_{\ell} I - H_{j}  - K_{j}(z_{\ell})L_{j}(z_{\ell})M_{j}(z_{\ell} \overline{I} - \overline{H}_j)) \vy = V^{*}_{j} (I-Q) \vb$ for $\vy$ 
} 
\STATE{$\vt_{1} \leftarrow \vt_{1} + \mu \hspace{0.1cm} \vy$}
\STATE{$\vt_{2} \leftarrow \vt_{2} + \mu \hspace{0.1cm} (z_\ell I - U^{*} C)^{-1}(U^{*} \vb - U^{*} V_{j+1}(z_\ell \overline{I} - \overline{H}_{j}) \vy)$ \label{line:t2_update}}
\ENDFOR
\STATE{$\widetilde{\vf}_{1} \leftarrow V_{j} \vt_{1} + U \vt_{2}$}
\STATE{Construct new orthonormal $U$,  and $C$ such that $C = A U$.}
\end{algorithmic}
\end{algorithm}

Table \ref{tab:version1_work} gives details on the arithmetic cost of an efficient implementation of the algorithm, excluding the matrix-vector products with $A$. It reports the highest order leading term in the respective arithmetic work in units of one multiplication plus one addition. We do not count matrix scalings, matrix additions and matrix-vector products since their cost is of lower order than that for matrix products and inversions. We also neglect the QR factorization required to obtain the orthogonal columns in  $U$ since this is common to all methods. We treat the Hessenberg matrices as full matrices, and we assume that we evaluate the matrix products in \eqref{eq:matrix-product} from right to left,
and that we compute the inverse $L(\sigma)$ explicitly (with work $k^3$). Note that we don't count any work for line~\ref{line:t2_update}, since the computations there can be arranged as matrix-vector products (and vector additions) only.

\begin{table}
\centering
\begin{tabular}{|l|c|c|} \hline
operation & arithmetic cost & how often \\ \hline 
$U^*C$   &  $k^2n$ & 1  
\\
$V_{j}^{*} U$ and $V_{j}^{*} C$ & $2 k j n $ & 1
\\
$M_{j} = U^{*} V_{j+1}$ & $k (j+1) n$ & 1
\\
$M_{j} (z_\ell I - \overline{H}_{j}) := R(z_\ell) $ & $k j (j+1)$  & $n_{quad}$
\\
$L_{j}(z_\ell) = (z_\ell U^{*} U - U^{*} C)^{-1} $ & $k^3$ & $n_\quadp$ 
\\
$L_{j}(z_\ell) R(z_\ell) =: S(z_\ell)$ & $k^{2} j$ & $n_{quad}$
\\
$K_{j}(z_\ell)S(z_{\ell}) =: T(z_\ell)$ & $k j^{2}$ & $n_{quad}$
\\
Solve for $\vy$, system matrix is $T(z_\ell)$ & $\frac{1}{3} j ^{3}$ & $n_{quad}$
\\
\hline \hline
\multicolumn{3}{|l|}{\textbf{total cost: $n (k^{2} + 2 j k + k(j+1)) + n_{quad}(k j (j+1) + k^{3} + k^{2} j + k j^{2} + \frac{1}{3} j^{3})$} }
\\ 
\hline
\end{tabular}
\caption{Arithmetic cost of rFOM (quad) version~1. \label{tab:version1_work}}
\end{table}
%omit version 2 and put in the appendix

\subsection{rFOM (quad) version ~2}\label{sec:version-2}
We present Algorithm \ref{version2:alg} as a potential alternative to Algorithm  \ref{version1:alg}. It treats the contribution from the Krylov subspace and the augmentation space as a single unknown by making use of an augmented Arnoldi relation as shown in the following proposition. 

\begin{proposition}
 For the choices of subspaces  $\widetilde{\mathcal{V}}_{j} = \mathcal{V}_{j} = \mathcal{K}_{j}(A,\vb)$ and $\widetilde{\mathcal{U}} = \mathcal{U}$, the augmented Krylov subspace approximation \eqref{eq:fullshiftedsolution}  
to the shifted system 
\[
    (\sigma I-A)\vx = \vb
\]
can be written as
\begin{equation}\label{approximation2x}\normalfont
    \vx_{j}(\sigma) = \vV_{j} (\vV_{j}^{*}\vW_{j}(\sigma I - \vG_{j}) + \vV^{*}_{j} \vR_{\sigma})^{-1} \vV_{j}^{*}\vb,
\end{equation}
where 
\begin{equation*} \label{extended_Hessenberg:eq} \normalfont
\vG_{j} =  \begin{bmatrix} I & 0 \\ 0 & H_{j} \end{bmatrix} \in \mathbb{C}^{(k+j) \times (k+j)}, \enspace \vW_{j} = \begin{bmatrix} C & V_{j} \end{bmatrix}, \enspace  C = A U
\end{equation*}
and \normalfont
\begin{align*}\vR_{\sigma} = \begin{bmatrix} \sigma (U - C) & -h_{j+1,j} \vv_{j+1} \ve_{j}^{T}
\end{bmatrix} \in \mathbb{C}^{n \times (k+ j)}.
\end{align*}
\end{proposition}
\begin{proof}
Using the fact that $C = A U$ we may write
\begin{eqnarray*}
    (\sigma I - A) U &=& C(\sigma I - I) + \sigma ( U - C),
\end{eqnarray*}
which, when combined with the Arnoldi relation (\ref{Arnoldi}) yields the shifted augmented Arnoldi relation
\begin{eqnarray}
    \lefteqn{(\sigma I - A) \begin{bmatrix} U & V_{j} \end{bmatrix}} &  & \nonumber \\
    & = & \begin{bmatrix}
       C & V_{j}
    \end{bmatrix} \begin{bmatrix}
       \sigma I - I & 0 \\
       0 & \sigma I - H_{j}
    \end{bmatrix} + \begin{bmatrix}
       \sigma (U - C) & -h_{j+1,j} \vv_{j+1} \ve_{j}^{T}
    \end{bmatrix}. \label{augmentedArnoldi}
\end{eqnarray}
The result now follows from \eqref{eq:fullshiftedsolutioncompact} by replacing the term $(\sigma I-A)\begin{bmatrix} U &V_{j} \end{bmatrix}$ with the right hand side of \eqref{augmentedArnoldi},
and setting $\widetilde{\vV}_{j} = \vV_{j}$ in \eqref{eq:fullshiftedsolutioncompact} since $\widetilde{U} = U$.
\end{proof}
Using \eqref{approximation2x} in the integral representation gives the approximation
\begin{equation} \label{approximation2}
    f(A)\vb \approx \vV_{j} \frac{1}{2 \pi i} \int_{\Gamma} f(z)   (\vV_{j}^{*} \vW_{j}(\sigma I - \vG_{j}) + \vV^{*}_{j} \vR_{\sigma})^{-1}    \d \sigma \vV_{j}^{*} \vb,
\end{equation}
which, using a quadrature rule, gives our 
second version of rFOM (quad) which computes the approximation $\tilde{\vf}_{2} \approx f(A)\vb$ via
\[
 \widetilde{\vf}_{2} := \vV_{j}  \sum_{\ell=1}^{n_{\quadp}} \omega_\ell f(z_{\ell})   (\vV_{j}^{*}\vW_{j}(z_{\ell} I - \vG_{j}) + \vV^{*}_{j} \vR_{z_{\ell}})^{-1}\vV_{j}^{*}\vb.
\]
using some quadrature procedure with nodes $z_{i}$ and weights $\omega_{i}$.

The second version of rFOM (quad) is presented in Algorithm \ref{version2:alg}, and we note we still prefer to impose orthonormality of $U$ (also in version ~3) to avoid numerical instabilities.

\begin{algorithm}
\caption{rFOM (quad) version ~2} 
\label{version2:alg}
\begin{algorithmic}[1]
\STATE{ \textbf{Input:} $A \in \mathbb{C}^{n \times n}$,  $\vb \in \mathbb{C}^{n}$, scalar function $f$, 
  $U \in \mathbb{C}^{n \times k} $ whose columns span the augmenting recycled subspace $\mathcal{U}$,  $C = A U$, Arnoldi cycle length $j$}
\STATE{Build a basis for $\mathcal{K}_{j}(A,\vb)$ via the Arnoldi process,  generating $V_{j+1}$ and $\overline{H}_{j}$}
\STATE{$\vV_{j} \leftarrow \begin{bmatrix}
    U & V_{j} 
 \end{bmatrix}$,
 \enspace 
 $ \vW_{j} \leftarrow \begin{bmatrix}
    C & V_{j} 
 \end{bmatrix}$}
 \STATE{ Determine contour containing the spectrum of  $A$ and determine quadrature rule for the integral in \eqref{approximation2} with nodes $z_\ell$ and weights $\omega_\ell, \ell = 1,\ldots,n_{\quadp}$}
\STATE{Set $\vt = \vnull \in \mathbb{C}^{j+k}$} 
\FOR{$\ell = 1,\ldots, n_{\quadp}$}
\STATE{$\mu \leftarrow \omega_\ell f(z_\ell)$}
\STATE{$\vt \leftarrow \vt + \mu \big(\vV_{j}^{*}\vW_{j}(z_{\ell} I - \vG_{j}) + \vV^{*}_{j} \vR_{z_{\ell}}\big)^{-1} \vV^{*}_{j} \vb $}
 \ENDFOR
\STATE{ $\widetilde{\vf}_{2} \leftarrow \vV_{j} \vt $}
\STATE{Construct new orthonormal $U$, and $C$ such that $C = A U$}
\end{algorithmic}
\end{algorithm}
The attractive feature of version~2 is that for each quadrature point only one inversion of a $(j+k) \times (j+k)$ matrix is needed. A cost summary of version $2$ is given in Table \ref{tab:version2_work}. We use the same approach as for the summary in Table \ref{tab:version1_work} and  count only matrix-matrix operations. We count the cost of constructing $\vV_j^{*} \vR_{z_{\ell}}$ in terms of the cost of constructing its block components
\[
   \vV^{*}_{j} \vR_{z_{\ell}} = \begin{bmatrix} z_{\ell} U^{*}(U - C) & U^{*} M \\
   z_{\ell} V_{j}^{*}(U - C) & V_{j}^{*} M
   \end{bmatrix}
\]
where $M = -h_{j+1,j} \vv_{j+1} \ve_{j}^{T}$. Since $M$ is a rank~1 matrix, products $U^{*} M$ and $ V^{*}_{j} M$
are obtained as matrix-vector and vector-vector operations and thus do not contribute to leading order terms in the arithmetic cost.

\begin{table}
\centering
\begin{tabular}{|l|c|c|} \hline
operation & arithmetic cost & how often \\ \hline 
$\vV_{j}^{*} \vW_{j}$ & $n (j+k)^{2}$ & 1
\\
$U^{*} (U - C)$ & $n k^{2}$ & 1
\\
$V_{j}^{*} (U - C)$ & $n k j$ & 1
\\
$(\vV_{j}^{*} \vW_{j}) (z_{\ell} I - \vG_{j}) $ & $(j+k)^{3}$ & $n_{quad}$
\\
Linear solve with $(\vV_{j}^{*} \vW_{j} (z_{\ell} I - \vG_{j}) + \vV_{j}^{*} \vR_{z_{\ell}})$ & $ \frac{1}{3}(j+k)^{3}$ & $n_{quad}$
\\
\hline \hline
\multicolumn{3}{|l|}{\textbf{total cost:} $n((j+k)^{2} +  k^{2} + k j) +  \frac{4}{3} n_{quad} (j+k)^{3}$ }
\\  \hline
\end{tabular}
\caption{Arithmetic cost of rFOM (quad) version~2. \label{tab:version2_work}}
\end{table}

\subsection{rFOM (quad) version ~3}
Our third version of rFOM (quad) is based on a reformulation of (\ref{approximation2}) which  involves an evaluation of $f$  on the extended Hessenberg matrix $\vG_{j}$ from \eqref{extended_Hessenberg:eq}. 
Approximation in terms of an evaluation of $f$ on some small Hessenberg matrix is typical of how standard Krylov subspace methods for matrix functions are implemented. This is the case for both the standard Arnoldi approximation, and the deflated Arnoldi like approximation in \cite{eiermann2011deflated}. As the reader can appreciate from \Cref{prop:contour-Gj}, we cannot write the full approximation completely in terms of $f(\vG_{j})$. We can, however, write it as the sum of a term involving $f(\vG_{j})$ and another integral expression. The purpose of this formulation is to introduce a term which can be evaluated without numerical quadrature. 
This takes away much of the dependence of the approximation on the quadrature and allows us to achieve a desired level of accuracy with less quadrature points as shown in the numerical experiments in Section \ref{numericalexperiments}. We derive version~3 of rFOM (quad) using the following proposition.
\begin{proposition}\label{prop:contour-Gj}
Assume that the contour $\Gamma$ not only encloses the spectrum of $A$, but also that of $\vG_{j}$, then the approximation (\ref{approximation2}) can equivalently be written as
\begin{equation} \label{approximation3:eq} \normalfont
    f(A)\vb \approx \vV_{j} f(\vG_{j}) (\vV_{j}^{*} \vW_{j})^{-1} \vV_{j}^{*} \vb - \vV_{j} \hspace{0.01cm} \mathcal{I} \hspace{0.01cm} \vV_{j}^{*} \vb
\end{equation}
with \normalfont
\begin{eqnarray*}
    \mathcal{I} &=& \frac{1}{2 \pi i} \int_{\Gamma}  f(\sigma) (\vV^{*}_{j} \vW_{j} (\sigma I - \vG_{j})  )^{-1} 
    S_j(\sigma)\vV_{j}^{*} \vR_{\sigma}( \vV^{*}_{j} \vW_{j} (\sigma I - \vG_{j})  )^{-1} \hspace{0.1cm} \d \sigma,
\end{eqnarray*}
where \normalfont $S_j(\sigma) \; = \; (I + \vV_{j}^{*} \vR_{\sigma}( \vV^{*}_{j} \vW_{j} (\sigma I - \vG_{j})  )^{-1} )^{-1}$.
\end{proposition}

\begin{proof}
By the Sherman–Morrison–Woodbury identity we have
\begin{eqnarray*}
\lefteqn{(\vV_{j}^{*}\vW_{j}(\sigma I - \vG_{j}) + \vV^{*}_{j} \vR_{\sigma})^{-1} }  &  & \\
& = & ( \vV^{*}_{j} \vW_{j} (\sigma I - \vG_{j})  )^{-1}   - ( \vV^{*}_{j} \vW_{j}(\sigma I - \vG_{j}) )^{-1}
S_j(\sigma) \vV_{j}^{*} \vR_{\sigma}( \vV^{*}_{j} \vW_{j} (\sigma I - \vG_{j})  )^{-1}.
\end{eqnarray*}
Abbreviating $
 G_{\sigma,j} :=  \vV^{*}_{j} \vW_{j} (\sigma I - \vG_{j})$ 
and substituting into (\ref{approximation2})  yields
\begin{eqnarray*}
 f(A)\vb &\approx&  \vV_{j} \frac{1}{2 \pi i}\int_{\Gamma}  f(\sigma) (\sigma I - \vG_{j})^{-1} \d \sigma  \hspace{0.1cm}(\vV^{*}_{j} \vW_{j})^{-1} \vV_{j}^{*} \vb \\   & & \mbox{} -  \vV_{j} \frac{1}{2 \pi i} \int_{\Gamma} f(\sigma)  G_{\sigma,j}^{-1} (I + \vV_{j}^{*} \vR_{\sigma}  G_{\sigma,j}^{-1} )^{-1} \vV_{j}^{*} \vR_{\sigma}  G_{\sigma,j}^{-1} \vV^{*}_{j} \vb \hspace{0.1cm} \d \sigma,
\end{eqnarray*}
which is  the desired result.
\end{proof}
Version~3 of rFOM (quad) uses numerical quadrature for evaluation of $\mathcal{I}$ in \eqref{approximation3:eq} giving the approximation 
\begin{align*}
 \widetilde{\vf}_{3} :=  &\vV_{j} f(\vG_{j}) (\vV_{j}^{*} \vW_{j})^{-1} \vV_{j}^{*} \vb  \\
 & - \vV_{j} \sum_{\ell = 1}^{n_{\quadp}} \omega_\ell f(z_{\ell})  G_{z_{\ell},j}^{-1} (I + \vV_{j}^{*} \vR_{z_{\ell}} G_{z_{\ell},j}^{-1})^{-1} \vV_{j}^{*} \vR_{z_{\ell}} G_{z_{\ell},j }^{-1} \vV_{j}^{*} \vb.
\end{align*} 

\begin{algorithm}
\caption{rFOM (quad) version ~3 \label{version3:alg}  } 
\begin{algorithmic}[1]
\STATE{\textbf{Input:} $A \in \mathbb{C}^{n \times n}$,  $\vb \in \mathbb{C}^{n}$, scalar function $f_{s}$, matrix function $f$, 
$U \in \mathbb{C}^{n \times k} $ whose columns span the augmenting recycled subspace $\mathcal{U}$,  $C = A U$, Arnoldi cycle length $j$}
\STATE{Build a basis for $\mathcal{K}_{j}(A,\vb)$ via the Arnoldi process, generating $V_{j+1}$ and $\overline{H}_{j}$}
\STATE{ $\vV_{j} \leftarrow \begin{bmatrix}
   U & V_{j} 
 \end{bmatrix}$,
 \enspace
 $ \vW_{j} \leftarrow \begin{bmatrix}
    C & V_{j} 
 \end{bmatrix}$ }
\STATE{Determine contour containing the spectrum of  $A$ and determine quadrature rule for the integral
$\mathcal{I}$ in \eqref{approximation3:eq} with nodes $z_{\ell}$ and weights $\omega_\ell, \ell = 1,\ldots,n_{\quadp}$}
\STATE{Set $\vt = \vnull  \in \mathbb{C}^{j+k}$}
\FOR{$\ell = 1,\ldots, n_{\quadp}$}
\STATE{$\mu \leftarrow \omega_\ell f_{s}(z_{\ell})$}
\STATE{$\vt \leftarrow \vt + \mu \hspace{0.1cm} G_{z,j}^{-1} (I + \vV_{j}^{*} \vR_{z_{\ell}} G_{z,j}^{-1})^{-1} \vV_{j}^{*} \vR_{z_{\ell}} G_{z,j}^{-1} \vV_{j}^{*}\vb$}
\ENDFOR
\STATE{$\widetilde{\vf}_{3} \leftarrow  \vV_{j} f(\vG_{j}) (\vV^{*}_{j} \vW_{j})^{-1} \vV^{*}_{j} \vb - \vV_{j} \vt$ \label{line:mat_eval}}
\STATE{Construct new orthonormal $U$, and $C$ such that $C = A U$}
\end{algorithmic}
\end{algorithm}

An algorithmic description of rFOM (quad), version ~3, is given as Algorithm \ref{version3:alg} and an arithmetic cost summary is given in Table \ref{tab:version3_work}. The cost of evaluating $f(\vG_j)$ will depend on whether we can use specifically tailored methods for specific functions $f$. Hence, we carry a term of $\mathcal{O}((j+k)^3)$ for that cost. The cost is proportional to $(j+k)^3$ if instead of a specific method we just use the eigendecomposition of $\vG_j$. The quantity $\vR_{z_\ell}$ is assumed to be computed as in version~2.

\begin{table}
\centering
\begin{tabular}{|l|c|c|} \hline
operation & arithmetic cost & how often \\ \hline 
$\vV_{j}^{*} \vW_{j}$ & $n (j+k)^{2}$ & 1
\\
Linear solve with $\vV_{j}^{*} \vW_{j}$ & $\frac{1}{3} (j + k)^{3}$ & 1
\\
$f(\vG_{j})$ & $\mathcal{O}( (j+k)^{3})$ & 1
\\
$U^{*} (U - C)$ & $n k^{2}$ & 1
\\
$V_{j}^{*} (U - C)$ & $n j k$ & 1
\\
$\vV_{j}^{*} \vW_{j} (z I - \vG_{j}) := G_{z,j}$ & $(j+k)^{3}$ & $n_{quad} $
\\
$G_{z,j}^{-1} $ & $(j+k)^{3}$ & $n_{quad}$
\\
$\vV_{j}^{*} \vR_{z_{\ell}} G_{z,j}^{-1}$ & $(j+k)^{3}$ & $n_{quad}$
\\
Linear solve with $ (I + \vV_{j}^{*} \vR_{z_{\ell}} G^{-1}_{z,j} ) $ & $\frac{1}{3}(j+k)^{3}$ & $n_{quad}$
\\
\hline \hline
\multicolumn{3}{|l|}{\textbf{total cost: $n ((j+k)^{2} +  k^{2} + jk) + n_{quad} \frac{10}{3} (j+k)^{3} + \mathcal{O}((j+k)^3)$}}
\\  \hline
\end{tabular}
\caption{Arithmetic cost of rFOM (quad) version~3. \label{tab:version3_work}}
\end{table}

We summarize the cost analysis in this section as follows: aside from the cost of the Arnoldi process, the costs of all three versions have leading terms which are $n$ times a sum of quadratic terms in $k$ and $j$ plus $n_\quadp$ times a sum of cubic terms in $k$ and $j$. 
In this sense all versions have comparable cost, although the constants involved in the quadratic and cubic terms differ from one version to the other.
For all versions, the $n_{quad}$ term in the cost is not dependent on the dominant $n$ and thus one can increase the number of quadrature points without causing large growths in cost.

\section{Extension to Stieltjes functions}\label{sec:stieltjes}
A Stieltjes function $f$ is a function which admits an integral representation
\[
f(z) = \int_{-\infty}^0 \frac{g(\sigma)}{\sigma-z} \d \sigma, \enspace z \in \mathbb{C} \setminus (-\infty,0]
\]
where $g$ is a function which has constant sign in $[0,\infty).$ Important Stieltjes functions are the inverse powers $f(z) = z^{-\alpha}$ with $\alpha \in (0,1)$, as well as other examples such as $f(z) = \frac{\mbox{log}(1-z)}{z}$: see \cite{frommer2014efficient, Henrici77}. 

The integral representation of a Stieltjes matrix function
\[
f(A)  = \int_{-\infty}^0 g(\sigma)(\sigma I-A)^{-1} \d\sigma
\]
resembles the Cauchy integral representation \eqref{integral_representation:eq} in that it also relies on the resolvent $(\sigma I-A)^{-1}$, the difference being that instead of a contour $\Gamma$ we now have a fixed integration interval $(-\infty,0]$. Thus we can also use all three versions of
rFOM (quad) on Stieltjes functions with the obvious change in the domain of integration.
Typical quadrature rules for the interval $(-\infty,0]$ first map the infinite interval to a finite interval, $[-1,1)$, say, and then use (variants of) Gaussian quadrature; see \cite{frommer2014efficient}, e.g.

\section{Updating the recycled subspace}\label{harmritz}
A practical implementation of rFOM (quad) requires an appropriate choice of the augmentation/recycled subspace $\mathcal{U}$. For most functions of practical interest (e.g. log, sign, inverse) the singularity is at the origin. We thus restrict our attention to these types of functions and aim to recycle $k$ approximate eigenvectors corresponding to the $k$ smallest eigenvalues of the previous matrix in the sequence. We shortly present how to obtain the Ritz approximations.

% \begin{proposition}
%     The Ritz problem for the matrix $A$ with respect to the augmented Krylov subspace space $\mathcal{K}_{j}(A, \vb) +  \mathcal{U}$ is given by the following eigenproblem \normalfont
%      \begin{equation}\label{eq:rits_problem_for_fab}
%    \mu \begin{bmatrix}
%        I & U^{*}\textbf{V}_{j} \\ \textbf{V}_{j}^{*} U & I
%    \end{bmatrix} \vg = \begin{bmatrix}
%        U^{*} C & U^{*}\textbf{V}_{j+1}\overline{H}_{j} \\
%        \textbf{V}_{j}^{*}C & H_{j}
%    \end{bmatrix} \vg
% \end{equation}
% for the eigenpair ($\vg$, $\mu$) with $\vg \in \mathbb{C}^{k+j}$ and $\mu \in \mathbb{C}$.
% \end{proposition}

% \begin{proof}
% The eigenproblem (\ref{eq:rits_problem_for_fab}) arises from applying the Ritz projection
% \[
%     \mbox{Find} \hspace{0.4cm} \vy \in  \mathcal{U} + \mathcal{K}_{j}(A,\vb) \hspace{0.4cm} \mbox{such that} \hspace{0.4cm} A \vy - \mu \vy \perp \mathcal{U} + \mathcal{K}_{j}(A,\vb) \hspace{0.4cm} \mbox{for some} \hspace{0.1cm} \mu \in \mathbb{C}.
% \]
% The desired result arises using the orthonormality of the columns of $U$, and that we can write $\vy = \widehat{\textbf{V}}_{j} \vg $ since
% $\vy \in \mathcal{K}_{j}( A, \vb) + \mathcal{U}$.
% \end{proof}

\begin{proposition} If the basis $U$ of $\mathcal{U}$ is orthonormal, then the Ritz eigenpairs $(\vy,\mu)$ of the matrix $A$ with respect to the augmented Krylov subspace  $\mathcal{K}_{j}(A, \vb) +  \mathcal{U}$, defined by the conditions
    \begin{equation}\label{eq:harmritz_cond}
    \vy \in \mathcal{K}_{j}(A, \vb) +  \mathcal{U} \text { and } A\vy - \mu \vy \perp \mathcal{K}_{j}(A, \vb) +  \mathcal{U},
    \end{equation}
    are given by $(\vV_j \vg,\mu)$ where $(\vg,\mu)$ is a solution of the generalized eigenproblem \normalfont
     \begin{equation}\label{eq:rits_problem_for_fab}
   \mu \begin{bmatrix}
       I & U^{*} V_{j} \\ V_{j}^{*} U & I
   \end{bmatrix} \vg = \begin{bmatrix}
       U^{*} C & U^{*} V_{j+1}\overline{H}_{j} \\
       V_{j}^{*}C & H_{j}
   \end{bmatrix} \vg
\end{equation}
in $\mathbb{C}^{k+j}$.
\end{proposition}

\begin{proof}
Clearly, since $\vy \in \mathcal{K}_{j}(A,\vb) + \mathcal{U}$, we can write $\vy = \vV_{j} \vg$ for some $\vg \in \mathbb{C}^{j+k}$.
The eigenproblem (\ref{eq:rits_problem_for_fab}) arises from application of the orthogonality condition in \eqref{eq:harmritz_cond} yielding
\[
\vV_j^*(A \vV_j \vg - \mu \vV_j \vg) = 0,
\]
and  using $U^*U = I$, $ V_j^* V_j = I$, $AU = C$ and the Arnoldi relation \eqref{Arnoldi}.
\end{proof}

%\begin{proposition}
%The harmonic Ritz problem for the matrix $A$ with respect to the augmented Krylov subspace space $\mathcal{K}_{j}(A, \vb) +  \mathcal{U}$ is given by the following eigenproblem  \normalfont 
%\begin{equation}\label{eq:harm_ritz_prob_for_fab}
%   \mu \begin{bmatrix}
%       C^{*}U & C^{*}\textbf{V}_{j} \\
%       \overline{H}_{j}^{*}\textbf{V}_{j+1}^{*} U & H_{j}
%   \end{bmatrix} \vg = \begin{bmatrix}
%       C^{*} C & C^{*} \textbf{V}_{j+1} \overline{H}_{j} \\
%       \overline{H}_{j}^{*}\textbf{V}_{j+1}^{*} C & \overline{H}_{j}^{*} \overline{H}_{j}
%   \end{bmatrix} \vg
%\end{equation}
%for the eigen-pair ($\vg$, $\mu$) with $\vg \in \mathbb{C}^{k+j}$ and %$\mu \in \mathbb{C}$.
%\end{proposition}
%\begin{proof}
%This proof follows the same arguments as the Ritz case, except we instead apply the harmonic Ritz projection
% \begin{align*}
%  \mbox{Find} \hspace{0.3cm} \vy \in  \mathcal{U} + \mathcal{K}_{j}(A,\vb) \hspace{0.3cm} \mbox{such that} \hspace{0.3cm} A \vy - \mu \vy \perp A( \mathcal{U} + \mathcal{K}_{j}(A,\vb)) \hspace{0.3cm} \mbox{for} \hspace{0.1cm} \mu\in \mathbb{C}.
%\end{align*} 
%\end{proof}
After we compute a set of $k$ Ritz vectors $\vy_{i}, i = 1,2, \ldots, k$ we can then store these vectors as the columns of the matrix $U$
and construct $C$ accordingly via $k$ matrix-vector products with the matrix $A$ for the next problem in the sequence.

\section{Numerical experiments}\label{numericalexperiments}
In this section we present results of numerical experiments. We demonstrate the effectiveness of rFOM (quad) as a recycling method for solving a sequence of problems using a comparison to the standard FOM  approximation. We drop the  ``(quad)''  in the name on the plot legends for rFOM (quad). All experiments have been performed in MATLAB, and the code \footnote{Code available at \url{https://github.com/burkel8/MatrixFunctionRecycling }} uses built-in MATLAB functions for the direct evaluation of matrix functions in each of the approximations. These MATLAB functions are also used to obtain the ``exact'' solution for each experiment which we used to determine the relative error
\[
\frac{\|f(A^{(i)})\vb^{(i)} - \vf_m^{(i)}\|}{\|f(A^{(i)})\vb^{(i)}\|},
\]
where $\vf_{m}^{(i)}$ is the rFOM (quad) approximation for $f(A^{(i)}) \vb^{(i)}$.

Our experiments are performed using a QCD matrix $Q_{\mu}$ of size $3,072 \times 3,072$ constructed from an initial $4^4$ configuration, as above, with $\kappa_c = 0.137$ and chemical potential $\mu = 0.3 $. 

%To build the sequence, each gauge link $U_\nu(\ell)$ of the previous configuration is randomly perturbed to give a new SU(3)-matrix with the difference to the old gauge link being controlled by a parameter $\delta \in [0,1]$. The smaller $\delta$ the smaller the perturbation with $\delta = 0$ corresponding to no perturbation at all.   

% do perturbations of the form $D^{(i+1)}_{\mu} = D^{(i)}_{\mu} + E^{(i)}_{\mu}$ where $E^{(i)}_{\mu}$ is a small variation of the physical objects from which the Wilson-Dirac operator is built, and by doing so, we keep the symmetries of $D_{\mu}$ intact throughout the sequence. We use a parameter $\delta$ to set the ``closeness" of $D_{\mu}^{(i)}$ to $D_{\mu}^{(i+1)}$.
%By writing $D_{\mu}(\kappa) = I - \kappa D_{H,\mu}$, the operator $D_{\mu}$ that we use in our numerical experiments is $D_{\mu}(\frac{4}{3}\kappa_{c})$.\todo{GR: it'd be good to add a reference here on $\kappa_{c}$, and also, how much of an explanation do we need to add here regarding $D_{\mu}(\frac{4}{3}\kappa_{c})$?}

Since the sign function is typically computed using the relation $\sign(Q_{\mu}) = (Q_{\mu}^{2})^{-1/2} Q_{\mu}$, our experiments actually report the results for the inverse square root function for $Q_\mu^2$, which can be represented as a Stieltjes integral 
\[
z^{-1/2} = \frac{-1}{\pi} \int_{-\infty}^{0} \frac{\sigma^{-1/2}}{\sigma-z} \d \sigma,
\]
for which we use a variant of Gaussian quadrature as detailed in \cite{frommer2014efficient}.

In the first experiment, we begin with the matrix $Q_{\mu}$ (as above) and a randomly generated vector $\vb$, and compute the inverse square root of $Q_{\mu}^{2}$ applied to $Q_{\mu}\vb$ using all three implementations of rFOM (quad). The augmentation subspace used is constructed from high quality eigenvector approximations computed directly from the matrix $Q_{\mu}^{2}$. We keep the number of Arnoldi iterations fixed at $m = 50$, the recycling subspace dimension fixed at $k = 20$. The Arnoldi method thus extracts an approximation from a subspace of dimension $m$ while the recycling methods use a subspace of dimension $m + k$. We vary the number of quadrature nodes used to compute the approximation. In Figure \ref{fig:quad_test} we plot the relative error in rFOM (quad) obtained after the full $m$ Arnoldi iterations, and compare this error to that obtained from the standard closed-form Arnoldi approximation (FOM) \eqref{Arnoldi_Cauchy:eq}, as well as a quadrature-based implementation (FOM quad) of the standard Arnoldi approximation.
\begin{figure}%[H]
    \centering
    \includegraphics[trim=0.5cm 9cm 0.5cm 9cm, width= .8\textwidth]{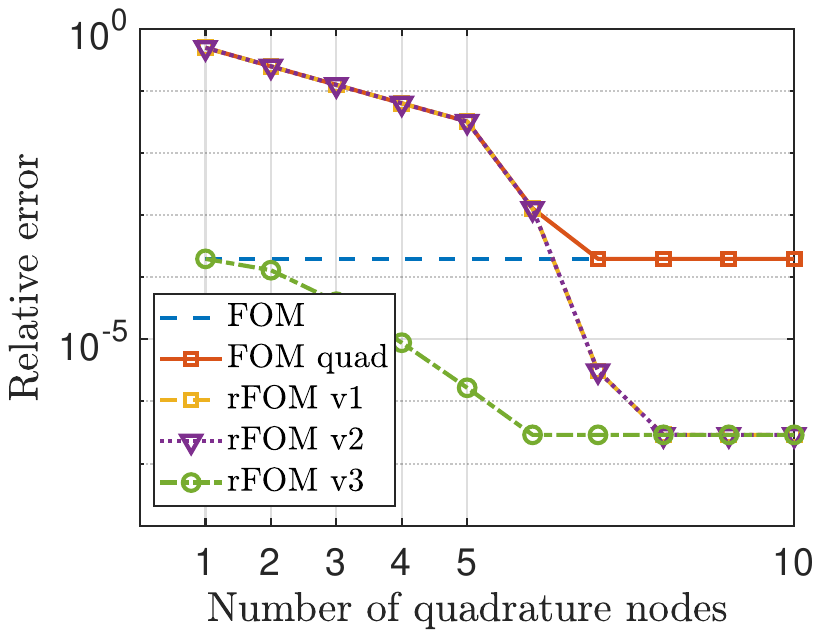}
    \caption{inverse square root function: relative error of rFOM (quad) for fixed $m = 50$ and $k = 20$ versus the number of quadrature nodes used.}
    \label{fig:quad_test}
\end{figure}

As expected, we need a sufficient number of quadrature nodes, particularly for versions $1$ and $2$ in order for rFOM (quad) to be at least as accurate as the standard FOM approximation. Note that versions 1 and 2 behave very similarly which is why the curve for version 1 is hidden behind that  of version 2 in the figure.  Beyond a certain number of quadrature points, the error of the quadrature Arnoldi stagnates and being identical to the error of standard Arnoldi, indicating that the quadrature error has become negligible as compared to the subspace approximation error. For the augmented methods, we see more  quadrature points are needed before such a stagnation occurs, and we note that version $3$ requires the least amount of quadrature points for this stagnation to occur. This is because this approximation involves a term which does not depend on quadrature but a ``closed form'' evaluation of the function.  Since versions ~1 and ~2 of rFOM (quad) are based entirely on quadrature, both approximations require a larger number of quadrature points to achieve the same level of accuracy as version ~3. At stagnation and beyond, the rFOM (quad) approximations  significantly outperform the standard Arnoldi approximation and all three implementations produce the same approximation. 

We now repeat a similar experiment using a sequence of problems with the same matrix $Q_{\mu}$, but different independent random vectors $\vb^{(i)}$. We run four experiments with different values of $m$, and a fixed recycling subspace dimension of $k=30$, and plot the relative error obtained after $m$ iterations of each method, for all problems in the sequence.
For all our experiments we make no assumption on the relationship between consecutive vectors $\vb^{(i)}$ and generate an independent random vector for each new problem in the sequence. The subspace $\mathcal{U}$ has been computed between each system using the Ritz procedure outlined in Section~\ref{harmritz}, and enough quadrature points have been chosen to ensure that the standard Arnoldi and the quadrature based Arnoldi produce the same approximation. Thus we only include the standard Arnoldi approximation (FOM) in the results in Figure \ref{fig:rfomquad_fixedrhs}. 
 
\begin{figure}
\centering
  \begin{minipage}{\linewidth}
    \makebox[.5\linewidth]{\includegraphics[trim=0.5cm 9cm 0.5cm 9cm, width= .7\textwidth]{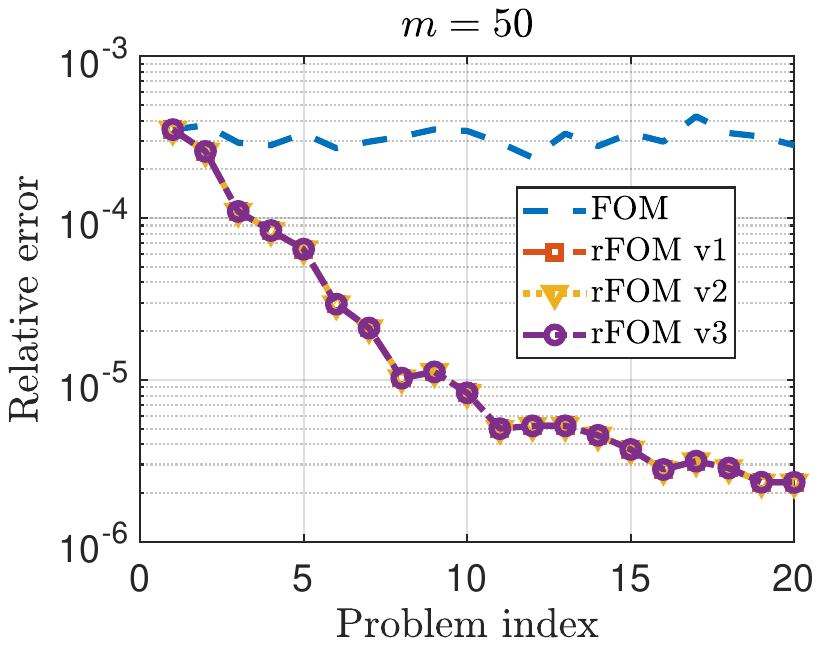}}%
    \makebox[.5\linewidth]{\includegraphics[trim=0.5cm 9cm 0.5cm 9cm, width= .7\textwidth]{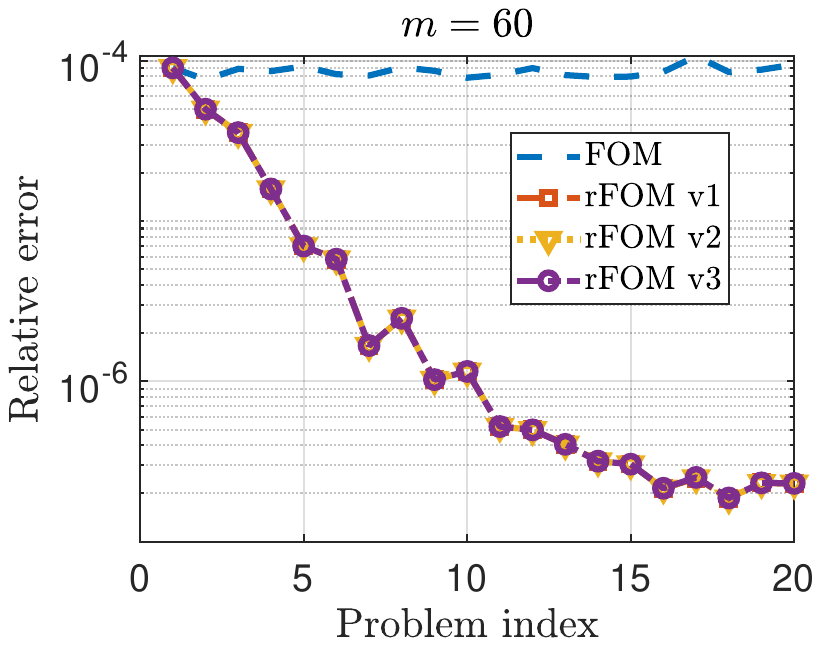}}

    \makebox[.5\linewidth]{\small (a)}%
    \makebox[.5\linewidth]{\small (b)}%

    \medskip

    \makebox[.5\linewidth]{\includegraphics[trim=0.5cm 9cm 0.5cm 9cm, width= .7\textwidth]{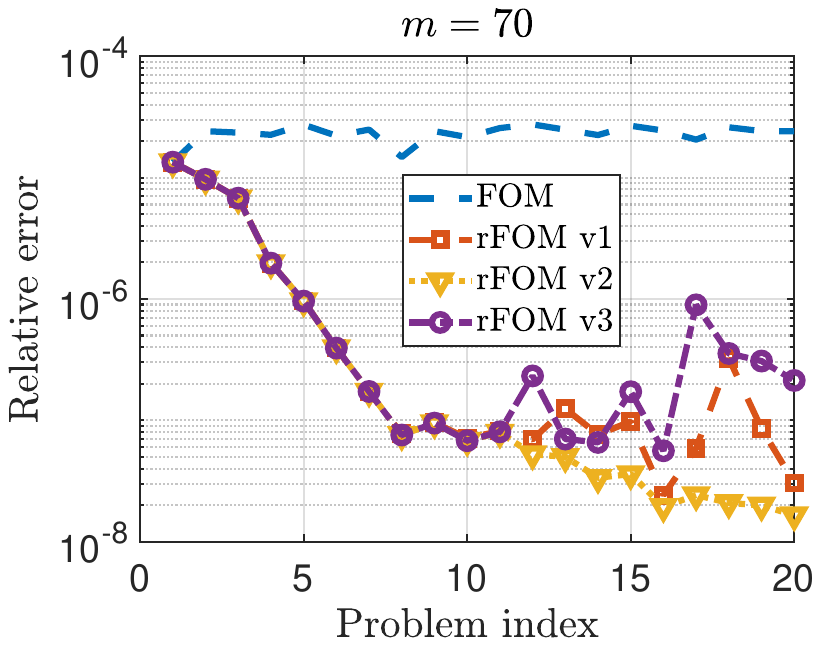}}%
    \makebox[.5\linewidth]{\includegraphics[trim=0.5cm 9cm 0.5cm 9cm, width= .7\textwidth]{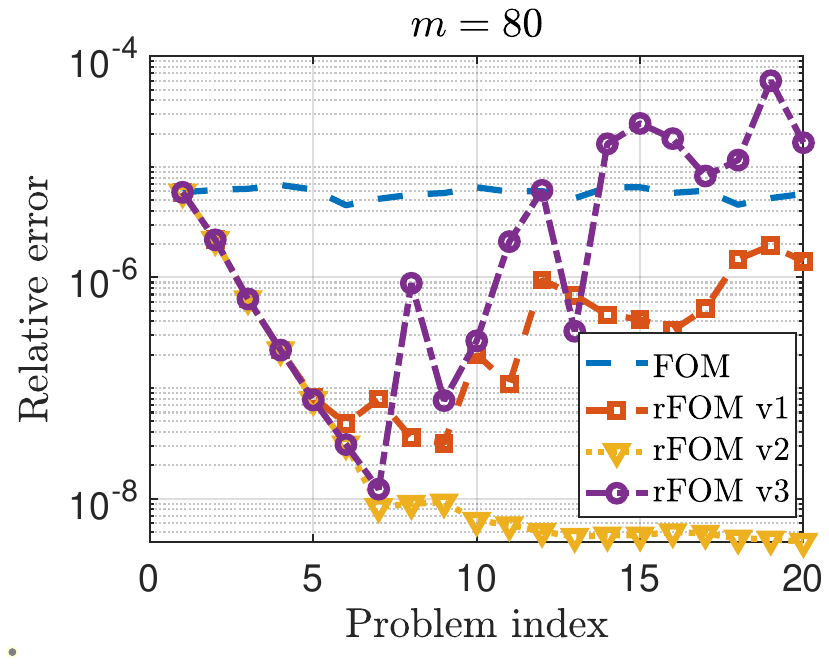}}

    \makebox[.5\linewidth]{\small (c)}%
    \makebox[.5\linewidth]{\small (d)}%
  \end{minipage}%
  \caption{inverse square root function: relative error for $20$ problems obtained at the end of $m$ Arnoldi iterations of rFOM (quad) for fixed $k = 30$. The plot shows the instability of rFOM (quad) for versions $1$ and $3$ for increasing values of $m$.}
  \label{fig:rfomquad_fixedrhs}
\end{figure}

From these experiments we see that for small enough $m$ all three variants of rFOM (quad) produce identical convergence curves. We see significant reductions in the relative error of rFOM (quad) for $m=50$ and $m=60$ as compared to the standard FOM approximation without recycling, and this error is consistently reduced as the sequence of problems progresses. This is one of the key qualities of an effective recycling method as outlined in \cite{parks2006recycling}.
However we note in subfigures (c) and (d) that versions $1$ and $3$ of rFOM (quad) suffer from stability issues for larger cycle lengths $m$, and so we believe version $2$ is the most robust implementation. From the remaining experiments we thus consider only version $2$ of rFOM (quad).
%which from here, we denote as `rFOM'.
%\todo{GR: this seems to disregard version 3 completely from hereon, but let's remember that version 3 might lead to a considerable reduction in the number of quadrature points needed, due to the closed-form nature of it as mentioned a couple of paragraphs ago. In this sense, combining this method with e.g.\ polynomial preconditioning might be helpful for version 3.}

In the final QCD experiment we consider the relative error curves obtained from solving a sequence of QCD problems with slowly changing matrices from four different sequences. Each sequence begins with the QCD matrix $Q_{\mu}$, and each matrix in the sequence is then constructed from the previous via random perturbations. In particular, each gauge link $U_\nu(\ell)$ of the previous configuration in the sequence is randomly perturbed to give a new SU(3)-matrix. The strength of the perturbation is controlled by a parameter $\delta \in [0,1]$, where smaller values of $\delta$ give rise to smaller perturbations, and $\delta = 0$ corresponds to no perturbation at all. In this experiment, each sequence differs by the value of $\delta$ used, and hence represents a different level of ``closeness'' between each of the problems. We take a fixed Arnoldi cycle length of $m = 50$ and a fixed recycling subspace dimension of $k=20$, and plot the relative error for each problem in the sequence. The results are shown in Figure \ref{fig:rfomquad_matchange}, and, as expected, we see that rFOM (quad) is most effective when the matrices in the sequence are close.

\begin{figure}
\centering
  \begin{minipage}{\linewidth}
    \makebox[.5\linewidth]{\includegraphics[trim=0.5cm 9cm 0.5cm 9cm, width= .7\textwidth]{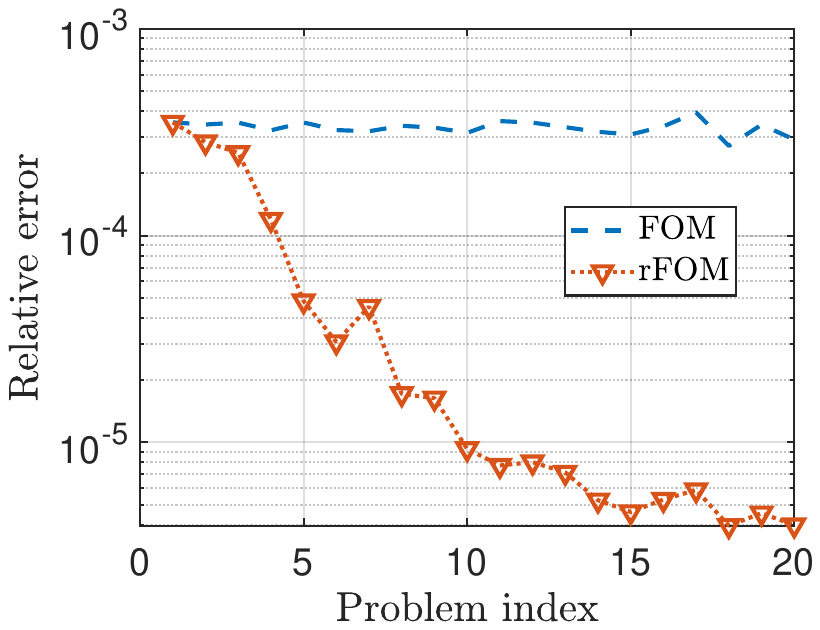}}%
    \makebox[.5\linewidth]{\includegraphics[trim=0.5cm 9cm 0.5cm 9cm, width= .7\textwidth]{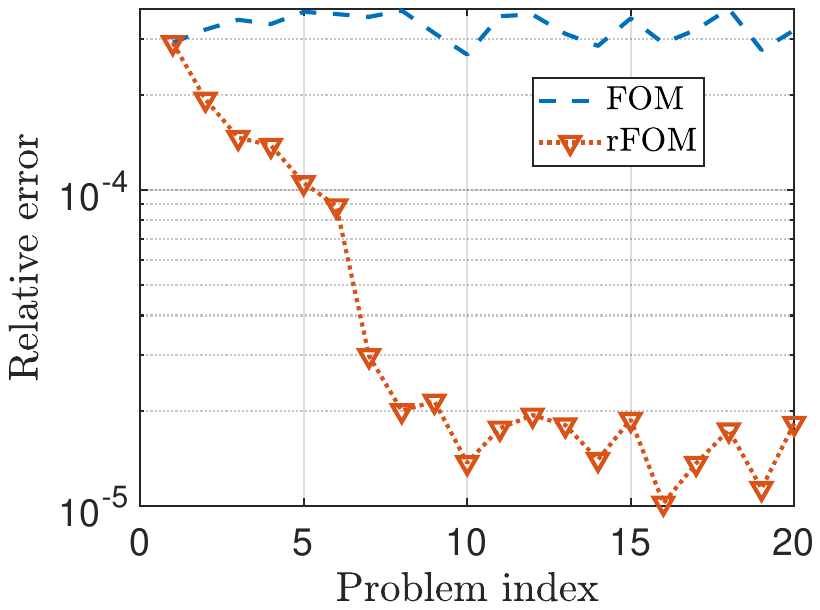}}

    \makebox[.5\linewidth]{\small (a)}%
    \makebox[.5\linewidth]{\small (b)}%

    \medskip

    \makebox[.5\linewidth]{\includegraphics[trim=0.5cm 9cm 0.5cm 9cm, width= .7\textwidth]{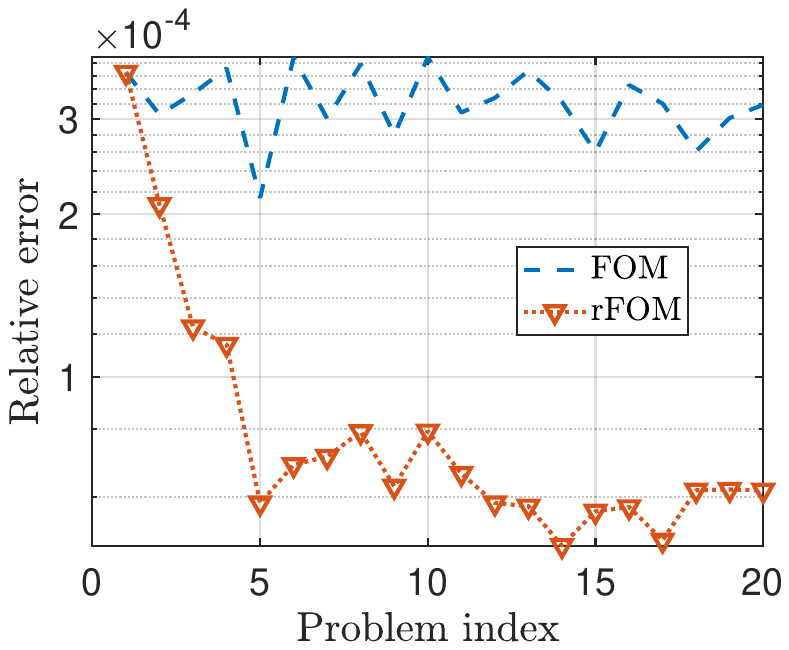}}%
    \makebox[.5\linewidth]{\includegraphics[trim=0.5cm 9cm 0.5cm 9cm, width= .7\textwidth]{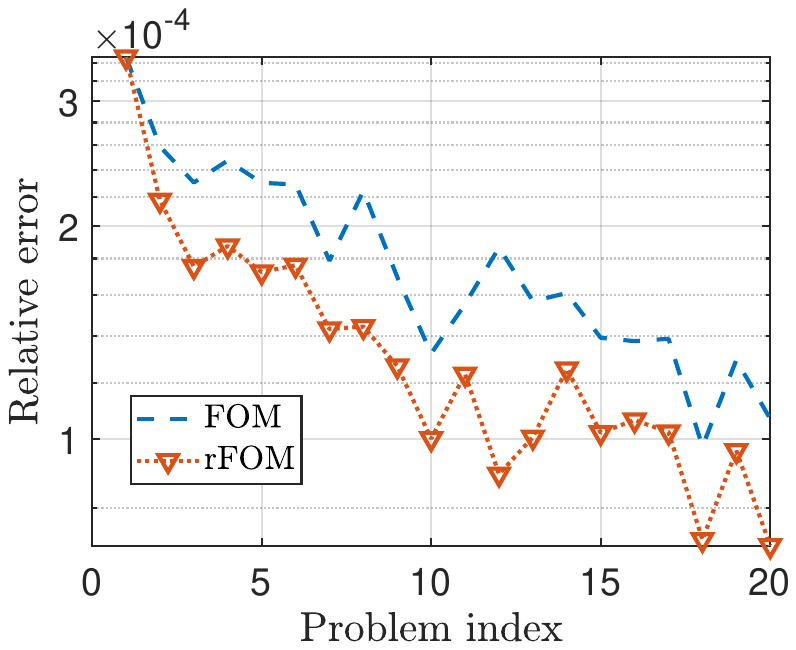}}

    \makebox[.5\linewidth]{\small (c)}%
    \makebox[.5\linewidth]{\small (d)}%
  \end{minipage}%
  \caption{inverse square root function: relative error for $20$ problems for different sequences with a fixed Arnoldi cycle length of $m = 50$ and recycling dimension $k = 20$. Each plot shows the relative error obtained for each problem in the sequence. In (a) $\delta = 0.0001$, (b) $\delta = 0.001$, (c) $\delta = 0.01$, (d) $\delta = 0.1$.}
   \label{fig:rfomquad_matchange}
\end{figure}

To demonstrate the generality of our approach, we end this section with a rather academic example using a different function and matrix. We consider $20$ applications of the log function on a matrix of size $1,600 \times 1,600$, representing a discretization of the 2d Laplace operator on a $40\times 40$ grid (obtained from the MATLAB matrix gallery), to a sequence of randomly generated vectors. We take $m = 50$ Arnoldi iterations with a recycling subspace dimension of $k = 20$ and plot the final error obtained for each problem. The result is shown in Figure \ref{fig:log_test}, where we observe a behaviour similar to the QCD inverse square root example. 

\begin{figure}[h]
    \centering
    \includegraphics[trim=0.5cm 9cm 0.5cm 9cm, width= .8\textwidth]{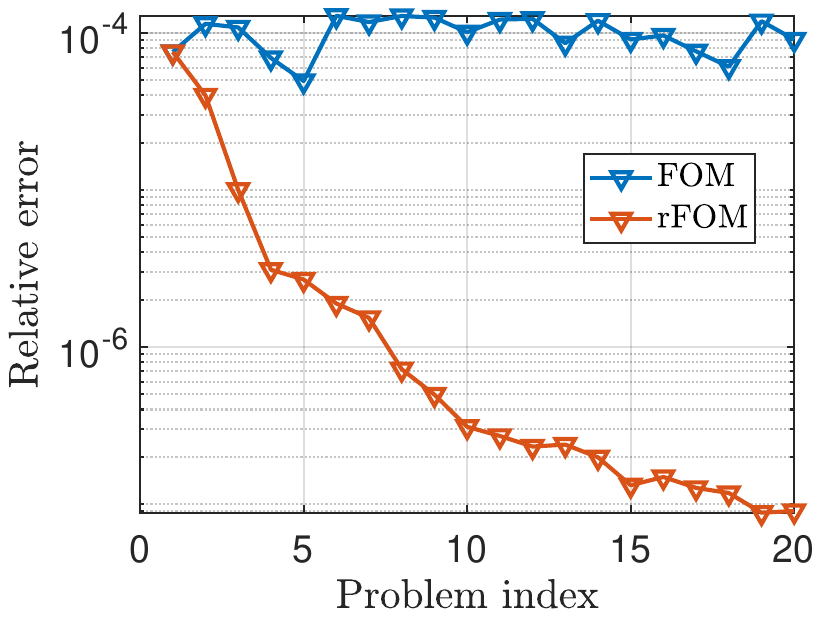}
    \caption{log function: relative error for $20$ problems with a fixed Arnoldi cycle length of $m = 50$ and recycling dimension $k = 20$.}
    \label{fig:log_test}
\end{figure}

\section{Conclusions and future work}\label{conclusions}
In this paper, we present the first augmented Krylov subspace method for the approximation of the action of a matrix function which allows for the recycling of \emph{arbitrary subspaces} $\mathcal{U}$ between a sequence of slowly changing matrices and vectors $f(A^{(i)})\vb^{(i)}$. We  presented a framework for understanding the development of such methods, but focused our attention on FOM-based recycling strategies. We presented three mathematically equivalent implementations of our approach with the same asymptotic costs, and we demonstrated that version~2 of the new method appears to be the most stable.  

In subsequent work for which the publication process is already completed \cite{burke2024krylov}, we have continued to build on the ideas and techniques presented in this paper in order to develop improved recycling algorithms for matrix functions, and address the stability issues associated to rFOM (quad) versions $1$ and $3$.

\section*{Acknowledgments}
We thank the anonymous reviewer for helpful suggestions which have improved the manuscript.
LB would like to thank the Hamilton Scholars and the Irish Research Council Government of Ireland postgraduate scholarship for funding this work, and Kathryn Lund for useful conversations. We thank Roman Höllwieser from University of Wuppertal for providing us the configurations used in this work. This work is partially supported by the German Research Foundation (DFG) research unit FOR5269 ``Future methods for studying confined gluons in QCD".

\bibliographystyle{unsrt}
\bibliography{references}

\end{document}